\documentclass[11pt]{amsart}
\usepackage[hmargin = 0.75 in, vmargin = 0.75 in]{geometry}
\usepackage{amssymb, amsmath,amsthm}
\newtheorem{thm}{Theorem}[section]

\newtheorem{cor}[thm]{Corollary}
\newtheorem{lem}[thm]{Lemma}
\theoremstyle{definition}

\newtheorem{def1}[thm]{Definition}

\newcommand{\ra}{\rightarrow}
\newcommand{\bk}{\backslash}
\newcommand{\mc}{\mathcal}
\newcommand{\mb}{\mathbb}
\newcommand{\sg}{\sigma}

\newcommand{\llf}{\left\lfloor}
\newcommand{\e}{\epsilon}
\newcommand{\rrf}{\right\rfloor}
\newcommand{\mbf}{\boldsymbol}
\begin{document}
\title[Restricted Prime Factors II]{On the Number of Restricted Prime Factors of an Integer II}
\normalsize
\author[A. P. Mangerel]{Alexander P. Mangerel}
\address{Department of Mathematics\\ University of Toronto\\
Toronto, Ontario, Canada}
\email{sacha.mangerel@mail.utoronto.ca}
\maketitle
\begin{abstract}
Given a partition $\{E_0,\ldots,E_n\}$ of the set of primes and a vector $\mbf{k} \in \mb{N}_0^{n+1}$, we compute an asymptotic formula for the quantity $|\{m \leq x: \omega_{E_j}(m) = k_j \ \forall \ 0 \leq j \leq n\}|$ uniformly in a wide range of the parameters $k_j$ that complements the results of a previous paper of the author \cite{prev}. This is accomplished using an extension and generalization of a theorem of Wirsing due to the author \cite{WirExt} that gives explicit estimates for the ratio $\frac{|M_g(x)|}{M_{f}(x)}$, whenever $f: \mb{N} \ra (0,\infty)$ and $g: \mb{N} \ra \mb{C}$ are strongly multiplicative functions that are uniformly bounded on primes and satisfy $|g(n)| \leq f(n)$ for every $n \in \mb{N}$. This also allows us to conclude the validity of a probabilistic heuristic set forth in \cite{prev} regarding $\pi(x;\mbf{E},\mbf{k})$ in a specific range of values of $k_j$, for each $0 \leq j \leq n$. 
\end{abstract}
\section{Introduction}
This is the second paper in a series in which we consider the following problem. Suppose $E_0,\ldots,E_n$ is a finite family of disjoint subsets of the primes and let $\mbf{k} \in \mb{N}_0^{n+1}$. We wish to determine 
\begin{equation} \label{GOAL}
\pi(x;\mbf{E},\mbf{k}) := |\{m \leq x: \omega_{E_j}(m) = k_j \ \forall 0 \leq j \leq n\}|,
\end{equation}
where, for each set $E_j$, $\omega_{E_j}(m)$ is the number of distinct primes from $E_j$ that divide $n$. Given $x \geq 3$ and a set of primes $E$, put
\begin{equation*}
E(x) := \sum_{p \leq x \atop p \in E} \frac{1}{p}.
\end{equation*}
A basic argument shows that $E(x)$ is the mean value of $\omega_{E}(n)$ for $n \leq x$, i.e., $x^{-1}\sum_{n \leq x} \omega_{E}(n) = \left(1+o(1)\right)E(x)$. We will assume throughout this paper that, in the context of \eqref{GOAL}, $E_j(x)$ tends to infinity with $x$ for each $j$. \\
In \cite{prev}, we proved an asymptotic formula for \eqref{GOAL} under the assumptions that: i) the set $\{E_j\}_{0 \leq j \leq n}$ partitions the primes; ii) for each $j$, $E_j(x)^{2} \ll k_j \ll \log^{\frac{2}{3}-\e} x$, and some mild hypotheses on the distribution of \emph{some} element of the partition.
In spite of probabilistic heuristics proposed in \cite{prev} that suggested that $x^{-1}\pi(x;\mbf{E},\mbf{k})$ should asymptotically approximate the probability density function of a Poisson distributed random variable, in which case 
\begin{equation*}
x^{-1}\pi(x;\mbf{E},\mbf{k}) \text{ is asymptotically equal to } \prod_{j=0}^n \frac{E_j(x)^{k_j}}{k_j!}e^{-E_j(x)}, 
\end{equation*}
we proved a substantial bias away from this behaviour, at least in the regime given in ii). The purpose of this paper is to consider the same problem of estimating \eqref{GOAL}, again under assumption i), but in a range of values of $k_j$ with order of magnitude less than or equal to that of the asymptotic mean $E_j(x)$. We would like to determine whether or not the Poisson heuristic just described is accurate in this range of $k_j$. We emphasize here that the methods we use in this paper to treat this problem are completely different from and independent of those in \cite{prev}, and the arguments in this article can be read without reference to \cite{prev}.\\
Considering the analogous question for $m = 1$ (hence $E_1$ is the set of all primes, which we denote by $\mc{P}$) and writing $E_1(x)$ as $P(x)$, it is clear from results of Landau and Selberg, respectively (see the references for \cite{prev}) that when $k = o\left(P(x)\right)$, $x^{-1}\pi(x;\mc{P},k)$ is not asymptotically Poisson distributed,  while for $k = \left(1+o(1)\right)P(x)$, $x^{-1}\pi(x;\mc{P},k)$, it is. Indeed, denoting $\rho := \frac{k}{P(x)}$, one has
\begin{align} 
x^{-1}\pi(x;\mc{P},k) &= (1+o(1))\rho\frac{P(x)^k}{k!}e^{-P(x)} \label{Small},
\end{align}
for $x \geq 2$, where $\log_2 x = \log(\log x)$. The factor biasing $x^{-1}\pi(x;\mc{P},k)$ in \eqref{Small} is therefore $\frac{k}{P(x)}$, and clearly when $\rho = 1+o(1)$ the Poisson heuristic is valid. In analogy, Hal\'{a}sz \cite{Hal2} showed that in the case of a single set $E \subseteq \mc{P}$ and in the case $k = (1+o(1))E(x)$, $x^{-1}\pi(x;E,k)$ is also Poisson distributed.  
In this paper, we extend Hal\'{a}sz' result to a partition of multiple sets, i.e., in the case where $k_j = \left(1+o(1)\right)E_j(x)$ for each $j$, $x^{-1}\pi(x;\mbf{E},\mbf{k})$ is Poisson distributed (see Corollary \ref{CORWIRSING}). We will also prove asymptotic formulae for $\pi(x;\mbf{E},\mbf{k})$ for larger ranges of values of $k_j$, in each case attempting to show similar distributional behaviour to that implicit in \eqref{Small}.  \\
For a set of primes $E$, an arithmetic function $g : \mb{N} \ra \mb{C}$, $x,T > 0$ and $\tau \in \mb{R}$, write
\begin{align*}
D_E(x;g,\tau) &:= \sum_{p \leq x} \frac{1-\text{Re}(g(p)p^{-i\tau})}{p}, \\
\Delta_E(x;g,T) &:= \min_{|\tau| \leq T} D_E(x;g,\tau).
\end{align*}
\begin{def1}\label{NICE}
Let $g: \mb{N} \ra \mb{C}$ be an arithmetic function and let $\mbf{E} = (E_0,\ldots,E_n)$ be a partition of the primes. We say that the ordered pair $(g,\mbf{E})$ is a \emph{nice pair} if $\max_{0 \leq j \leq n} \Delta_{E_j}(x;g,T)/E_j(x) \gg_n 1$. \\ 
If, for each $j$, there is some $z_j \in \mb{C}$ such that $g(p) = z_j$ for each $p \in E_j$ then we will also write that $(\mbf{z},\mbf{E})$ is a nice pair. Finally, if $(\mbf{z},\mbf{E})$ is a nice pair for each $\mbf{z} \in \mb{C}$ then we will say that $\mbf{E}$ is a \emph{nice partition}.
\end{def1}
Note that the property that a partition is nice is specifically dependent on the distribution of the arguments of the complex numbers $z_j$ for each $j$. We employ this definition merely to avoid certain pathological sets $E_j$ that are constructed precisely to correlate with the argument of $g$ at its primes, e.g., $E_j := \{p : |\text{arg}(g(p)p^{-it})| < \e\}$, for $\e > 0$ chosen small.
\begin{def1} \label{GOOD}
We say that $\mbf{E}$ is a \emph{good partition} if, for each $1 \leq j \leq m$ there exists some $\lambda_j > 0$ such that the function $\zeta_{E_j}(s)-\lambda_j\log\left(1/(s-1)\right)$ is holomorphic in a neighbourhood of $s = 1$, for $\zeta_{E_j}(s) := \prod_{p \in E_j} \left(1-p^{-s}\right)^{-1}$.
\end{def1}
One of our main results (see also Theorem \ref{GOALTHM} below) in this paper is the following.
\begin{thm} \label{HalApp}
Let $x$ be sufficiently large and let $B = B(x) > 0$ be a non-decreasing function. For each $0 \leq j \leq n$, let $k_j \in \mb{N}$, $\delta_j(x) > 0$ be a non-increasing function of $x$, and set $\delta := \min_{0 \leq j \leq n} \delta_j$. Suppose that if $\rho_j := \frac{k_j}{E_j(x)}$ then we have $E_j(x)^{-1+\e} \ll \delta_j \leq \rho_j \leq B$ and $\frac{\delta}{B}\delta_jE_j(x) \ra \infty$ as $x \ra \infty$ for each $j$. For $\mbf{\rho} := (\rho_0,\ldots,\rho_n)$ put 
\begin{equation*}
f_{\mbf{\rho}}(m) := \prod_{0 \leq j \leq k} \rho_j^{\omega_{E_j}(m)}. 
\end{equation*}
Assume that $\mbf{E}$ is a nice partition. Then, provided the above assumptions hold, we have uniformly in $\delta_jE_j(x) \ll k_j \leq BE_j(x)$,
\begin{equation*}
\pi(x;\mbf{E},\mbf{k})  = \left(\sum_{m \leq x} f_{\mbf{\rho}}(m)\right)\prod_{0 \leq j \leq n} \frac{E_j(x)^{k_j}}{k_j!}e^{-k_j}\left(1+O_{B,n,\mbf{E}}\left(\left(\frac{B}{\delta}\right)^2\left(\left(\frac{\delta}{B}\delta_jE_j(x)\right)^{-\frac{1}{6}} + \log^{-\pi^3\delta/2} x\right)\right)\right). 
\end{equation*}
If $\mbf{E}$ is also a good partition then we have the sharper estimate
\begin{equation*}
\pi(x;\mbf{E},\mbf{k})  = \left(\sum_{m \leq x} f_{\mbf{\rho}}(m)\right)\prod_{0 \leq j \leq n} \frac{E_j(x)^{k_j}}{k_j!}e^{-k_j}\left(1+O_{B,n,\mbf{E}}\left(\left(\frac{B}{\delta}\right)^2\left(\left(\delta_j	E_j(x)\right)^{-\frac{1}{6}} + \log^{-\pi^3\delta/2} x\right)\right)\right).
\end{equation*}
\end{thm}
(The second estimate is sharper when the values of $\rho_j$ are substantially different, so that $\delta$ and $B$ may be forced to be quite dissimilar in order of magnitude.) \\
We discuss the proof of this result in the next section. We are able to show as a consequence that whenever the ratios $\rho_j$ in Theorem \ref{HalApp} are all sufficiently close to a common value $\rho$, then 
\begin{equation*}
\pi(x;\mbf{E},\mbf{k}) \sim \rho B(\mbf{\rho}) x\prod_{0 \leq j \leq n} \frac{E_j(x)^{k_j}}{k_j!}e^{-E_j(x)},
\end{equation*}
where $B(\mbf{\rho})$ is a bounded factor that approaches 1 as $\max_j|\rho_j-1| \ra 0$ (see Theorem \ref{GOALTHM}). When $\rho$ is sufficiently small, it turns out that we can actually give an asymptotic formula for $\pi(x;\mbf{E},\mbf{k})$ that is uniform in the range $E_j(x)^{\e} \ll k_j \ll E_j(x)^{1-\e}$ for each $j$ with this method. \\
The particular case of Hal\'{a}sz in which $\rho_j = 1+o(1)$ for each $j$, follows immediately, and is contained in the following.
\begin{cor}\label{CORWIRSING}
Suppose $\left|\frac{k_j}{E_j(x)}-1\right| \leq \eta$ for each $0 \leq j \leq n$, and that $\mbf{E}$ is a nice partition. Then
\begin{equation*}
\pi(x;\mbf{E},\mbf{k})  = x\prod_{0 \leq j \leq n} \frac{E_j(x)^{k_j}}{k_j!}e^{-E_j(x)}\left(1+O_{B,n}\left(E_j(x)^{-\frac{1}{6}} +\eta^{\frac{1}{2}} + e^{-\frac{1}{\sqrt{\eta}}}\right)\right). 
\end{equation*}
\end{cor}
As is the case with Hal\'{a}sz' theorem in \cite{Hal2} mentioned above, Theorem \ref{HalApp} and its corollaries are consequences of a more general analysis that finds its roots in the study of mean values of multiplicative functions, a classical subject in multiplicative number theory.  Let $g : \mb{N} \ra \mb{C}$ be an arithmetic function. Put 
\begin{equation*}
M_g(x) := \sum_{n \leq x} g(n). 
\end{equation*}
Suppose $g$ is a \emph{completely multiplicative} function, i.e., $g(mn) = g(m)g(n)$ for all $m,n \in \mb{N}$, with the following properties. First, assume there exists $\delta > 0$ with $\delta \leq |g(p)| \leq 2-\delta$ for all primes $p$. Second, suppose that there exists an angle $\theta \in [-\pi,\pi)$ and $\beta > 0$ for which $|\text{arg}(g(p))-\theta| \geq \beta$, where here and throughout this paper, $\text{arg}(z)$ is taken to be the branch of argument with cut line along the negative real axis.  Then Hal\'{a}sz proved the existence of a constant $c = c(\delta,\beta) > 0$ such that
\begin{equation}
x^{-1}M_g(x) \ll_{\delta,\theta,\beta} \exp\left(\sum_{p \leq x} \frac{|g(p)|-1}{p} - c\sum_{p \leq x} \frac{|g(p)|-\text{Re}(g(p))}{p}\right) \label{HALLLL}.
\end{equation}
In the specific situation in which $|g(p)-1|$ is uniformly small over all primes $p$ in a precise sense then
\begin{equation}
M_g(x) \sim x\exp\left(\sum_{p \leq x} \frac{g(p)-1}{p}\right) =: Ax \label{HALLL}
\end{equation}
Hal\'{a}sz used this last result to compute $x^{-1}\Pi(x;E,k)$ for $k = (1+o(1))E(x)$ (as mentioned above, and elaborated upon below), where $\Pi(x;E,k)$ is defined analogously to $\pi(x;E,k)$, but where $\omega_E(n)$ is replaced by $\Omega_E(n)$, the count, \emph{with multiplicity}, of the number of prime divisors from $E$ of an integer $n$. \\ 
Recall that an arithmetic function $g$ is \emph{strongly multiplicative} if: i) whenever $(m,n) = 1$, $g(mn) = g(m)g(n)$, and ii) $g(p^k) = g(p)$ for every $k \in \mb{N}$.  In \cite{WirExt} we proved an extension and generalization of Hal\'{a}sz' theorem with applications to (among other things) strongly multiplicative functions, which has the following weaker consequence. 
\begin{thm}[Theorem 2.2 in \cite{WirExt}] \label{HalGen}
i) Let $x \geq 3$ and $B>0$. Suppose $g(n)$ is a strongly multiplicative function, and let $\{E_1,\ldots,E_m\}$ be a partition of the set of primes such that for each $1 \leq j \leq m$ there are non-increasing functions $\delta_j = \delta_j(x) > 0$ and non-decreasing functions $B_j = B_j(x) > 0$ such that $\delta_j \leq |g(p)| \leq B_j \leq B$ uniformly over $p \in E_j$. Assume moreover that there are fixed $\phi_j \in [-\pi,\pi]$ and $\beta_j \in (0,\pi]$ such that for each $p \in E_j$, $|\arg(g(p))-\phi_j| \geq \beta_j$. Set $\gamma_{0,j} := \frac{27\pi\delta_j}{1024B_j}\beta_j^3$ and $\delta := \min_{1 \leq j \leq m} \delta_j$. Then for $T = \log^2 x$,
\begin{equation} \label{UPPER}
M_g(x)/M_{|g|}(x) \ll_{B,m,\mbf{\phi},\mbf{\beta}} \frac{B^2}{\delta}\exp\left(-\sum_{1 \leq j \leq m}B_j \left(\Delta_{E_j}(x;\tilde{g},T) - \sum_{p \leq x \atop p \in E_j} \frac{1-|\tilde{g}(p)|}{p}\right)\right),
\end{equation}
where for each $p \in E_j$, $\tilde{g}(p) := g(p)/B_j$. \\
If $\mbf{E}$ is a good partition and $\max_{1 \leq j \leq m} \Delta_{E_j}(x;g,T) \geq C\log_3 x$ for some $C > 0$ then, in fact,
\begin{equation*}
M_g(x)/M_{|g|}(x) \ll_{B,m,\mbf{\phi},C} \frac{B^2}{\delta}\exp\left(-\sum_{1 \leq j \leq m} c_j\sum_{p \in E_j \atop p \leq x} \frac{g(p)-|g(p)|}{p}\right),
\end{equation*}
where $c_j \geq \min_j\frac{\delta_j}{2B_j}\frac{\gamma_{0,j}}{\gamma_{0,j}+1}$. \\
ii) Suppose additionally that for each $j$ there is a number $\eta_j = \eta_j(x) \in (0,1)$ such that $|\text{arg}(g(p))| \leq \eta_j$ uniformly over $p \in E_j$. Set $A := \exp\left(\sum_{p \leq x} \frac{g(p)-|g(p)|}{p}\right)$, and let $\eta := \max_{1 \leq j \leq m} \eta_j$. Then
\begin{equation} \label{ASYMP}
M_g(x) = M_{|g|}(x)\left(A+O_{B,m}\left(\mc{R}\right)\right),
\end{equation}
where, for $d_1 := \delta^{-1}\sqrt{\eta}$ if $B\eta \leq \delta \leq \eta^{\frac{1}{2}}$ and $d_1 := 1$ otherwise,
\begin{equation*}
\mc{R} := \frac{B^2}{\delta}\left(\eta^{\frac{1}{2}}|A|\left(d_1 + \log^{-\frac{\delta}{3}} x + \delta^{-\frac{1}{2}}e^{-\frac{d_1}{\sqrt{\eta}}}\right) + |A|^{\frac{\gamma_0}{4(1+\gamma_0)}}\left(\log^{-\frac{\delta\beta^3}{2}} x + \delta^{-1}e^{-\frac{d_1}{\sqrt{\eta}}}\right)\right).
\end{equation*}
\end{thm}
We also proved the following theorem, which is a generalization of a theorem of Wirsing (for references, see \cite{WirExt}).
\begin{thm}[Theorem 2.5 in \cite{WirExt}] \label{WIRSINGEXT}
i) Let $g$ satisfy the hypotheses of Theorem \ref{HalGen} i). Suppose $f:\mb{N} \ra [0,\infty)$ is a multiplicative function satisfying $\delta \leq |g(p)| \leq f(p)\leq B$, $|g(n)| \leq f(n)$ for all $n$, and such that $\frac{M_g(x)}{M_{f}(x)} \ra 0$ as $x \ra \infty$. Then 
\begin{equation*}
\frac{|M_g(x)|}{M_{f}(x)} \ll_{B,\phi,\beta} \left(\frac{B}{\delta}\right)^3 \frac{|M_g(x)|}{M_{|g|}(x)}\exp\left(-\sum_{p \leq x} \frac{f(p)-|g(p)|}{p}\right),
\end{equation*}
and the appropriate estimate of the ratio $|M_g(x)|/M_{|g|}(x)$ in \eqref{UPPER} may be used. \\
ii) Suppose either $f$ and $g$ are both strongly multiplicative, and $g$ additionally satisfies the hypothesis $|g(p)-f(p)| \leq \eta$ for each prime $p$. Then if $A := \exp\left(-\sum_{p \leq x} \frac{|g(p)|-g(p)}{p}\right)$ and $X := \exp\left(-\sum_{p \leq x} \frac{f(p)-|g(p)|}{p}\right)$ then
\begin{equation*}
M_{g}(x) = M_{f}(x)\left(\exp\left(-\sum_{p \leq x} \frac{f(p)-g(p)}{p}\right) + O_{B,m}\left(\mc{R}_1|A| + \mc{R}_2X\right)\right),
\end{equation*}
where we have put
\begin{align*}
\mc{R}_1 := \left(\frac{B}{\delta}\right)^2 \left(X\left(d_1\eta^{\frac{1}{2}} + \log^{-\frac{\delta\beta^3}{2}} x + \delta^{-1}e^{-\frac{d_1}{\sqrt{\eta}}}\right) + \log^{-\frac{2\delta}{3}}x + \delta^{-1}e^{-\frac{2d_1}{\sqrt{\eta}}}\right). \\
\mc{R}_2 := \left(\frac{B}{\delta}\right)^2 \left(d_1\eta^{\frac{1}{2}}|A| + \log^{-\frac{2\delta}{3}}x + \delta^{-1}e^{-\frac{2d_1}{\sqrt{\eta}}} + |A|^{\frac{\gamma_0}{4(1+\gamma_0)}}\left(\log^{-\frac{\delta\beta^3}{2}} x + \delta^{-1}e^{-\frac{d_1}{\sqrt{\eta}}}\right)\right).
\end{align*}
\end{thm}
In the next section we will discuss how these theorems can be used to prove Theorem \ref{HalApp} and Theorem \ref{CORWIRSING}.
\section{Results and Strategy of Proof}
Given a complex vector $\mbf{z} \in \mb{C}^{n+1}$ with non-zero components, and $s \in \mb{C}$ with $\text{Re}(s) > 1$ we define 
\begin{align*}
F(\mbf{z};s) &:= \prod_{0 \leq j \leq n} \prod_{p \in E_j} \left(1+\frac{z_j}{p^s-1}\right) = \sum_{n \geq 1} \frac{f_{\mbf{z}}(n)}{n^s}.  
\end{align*}
where $f_{\mbf{z}}$ 
is implicitly defined to satisfy $f_{\mbf{z}}(n) = \prod_{0 \leq j \leq n} z_j^{\omega_{E_j}(n)}$. 
It is readily checked that $f_{\mbf{z}}(n)$ is accordingly a strongly multiplicative function.
Observe that for some choice of $\mbf{z}$ as above,
\begin{equation*}
\sum_{m \leq x} f_{\mbf{z}}(m) = \sum_{m \leq x} z_0^{\omega_{E_0}(m)}\cdots z_n^{\omega_{E_n}(m)} = \sum_{\mbf{k} \in \mb{N}_0^{n+1}}\pi(x;\mbf{E},\mbf{k}) z_0^{k_0}\cdots z_n^{k_n}.
\end{equation*}
Let $\rho_j > 0$ be arbitrary for each $j$, and set $\mc{D} := \prod_{0 \leq j \leq n} \partial D(0,\rho_j)$, the product set of bounding circles of each disk around 0 of radius $\rho_j$. Applying Cauchy's integral theorem in each of the $n+1$ variables $z_j$, we get
\begin{equation} \label{CAUCHYINT}
\pi(x;\mbf{E},\mbf{k}) = (2\pi i)^{-(n+1)} \int_{\mc{D}} \frac{dz_0 \cdots dz_n}{z_0^{k_0+1}\cdots z_n^{k_n+1}} \left(\sum_{m \leq x} f_{\mbf{z}}(m)\right).
\end{equation}
In \cite{prev} we approached the problem of estimating $\pi(x;\mbf{E},\mbf{k})$ by using the Perron formula (see section II.2 in \cite{Ten2}) to express $\sum_{m \leq x} f_{\mbf{z}}(m)$ as a line integral in $s$ of $F(\mbf{z};s)x^ss^{-1}$ and evaluating the multiple integral for $\pi(x;\mbf{E},\mbf{k})$ thus produced via the saddle point method. Let us observe that, in the present context, a saddle point argument, as implemented in \cite{prev}, is ineffectual in determining an asymptotic formula for \eqref{GOAL} (in the remainder of this paragraph we shall make frequent reference to particular elements of \cite{prev}; the reader may wish to skip this altogether if (s)he wishes). Indeed, the saddle point method requires crucially that in the Taylor expansion of a (multivariable) locally $\mc{C}^2$ function, the second derivative is negative and has sufficiently large absolute value; in this case, one can replace the truncated $\tau$-integral in Lemma 5.8 of \cite{prev} by the infinite integral of a Gaussian function with small error.  In Lemma 5.6 there it is shown that the size of the absolute value is bolstered by large values of $\rho_j := \frac{k_j}{E_j(x)}$ (see Lemma 4.2 there); here, in the regime where $k_j$ is typically much smaller than $E_j(x)$ or perhaps \emph{at most} a constant multiple of $E_j(x)$, $\rho_j$ (the value of which is essentially forced by Lemma 4.1 there) is no longer large enough to provide an asymptotic formula.  Thus, we must seek a new method. \\
It is sufficient to be able to give precise information about the summatory function $\sum_{m \leq x} f_{\mbf{z}}(m)$ for each $\mbf{z}$ with $|z_j| = \rho_j$ in order to evaluate $\pi(x;\mbf{E},\mbf{k})$ precisely.  In particular, it will be clear that the largest contributions to the integral take place when $|\text{arg}(z_j)|$ is rather small; hence, asymptotic information in this case, and sufficiently sharp upper bounds in all other cases are necessary in order to provide asymptotic information for the integral in \eqref{CAUCHYINT}.  \\
A direct application of \eqref{HALLLL} and \eqref{HALLL} (following the proof of Theorem \ref{HalApp} below) only provides asymptotic information for $f_{\mbf{z}}(p) = z_j$ in a small neighbourhood around 1; in practical terms related to our problem, we can only determine $\pi(x;\mbf{E},\mbf{k})$ for $k_j = (1+o(1))E_j(x)$. While this \emph{is} of interest to us, we would like to be able to estimate $\pi(x;\mbf{E},\mbf{k})$ when $\frac{k_j}{E_j(x)}$ is in fact much smaller, or a constant factor larger, than 1 as well.\\
A look at the proof in \cite{Hal2} shows that the term $A$ in \eqref{HALLL} arises, essentially, from a comparison of the Dirichlet series $F(\mbf{z};s)$ with $\zeta(s)$ directly; the condition $|f_{\mbf{z}}(p)-1| < \eta$ is then necessary to ensure that the ratio $\zeta(s)^{-1}F(\mbf{z};s)$ can be approximated by $A$ (which is independent of $\tau$), at least when $s = \sg + i\tau$ where $\sg$ is fixed to be slightly larger than 1 and $|\tau|$ is sufficiently small as $x \ra \infty$. 
Because $\sum_{m \leq x} f_{\mbf{z}}(m)$ is determined from $F(\mbf{z};s)$ by the Perron formula (which, for convenience, we will call the \emph{Perron transform} of $F(\mbf{z};s)$ with respect to $x$), and moreover it is typically the case (unless $f_{\mbf{z}}$ correlates with an archimedean character $n \mapsto n^{i\tau_0}$ for some $\tau_0 \in \mb{R}$) that the major contributions of the integrals of Perron transforms emerge when $|\tau|$ is small (in which case the factors $n^{-i\tau}$ have a longer period of oscillation throughout $n \leq x$, and less cancellation due to variation in its complex argument occurs), it follows heuristically that $F(\mbf{z};s)$ and $\zeta(s)$, up to the multiplicative factor $A$, should have the same Perron transform.  This is morally, the idea of Hal\'{a}sz' argument and, in fact, explains whence comes the factor $x$ (since the Perron transform of $\zeta(s)$ with respect to $x$ is $\sum_{n \leq x} 1 = x + O(1)$).  In a similar vein, if we wanted to consider cases where $z_j$ is not close to 1, it would be more favourable to compare $F(\mbf{z};s)$ to a Dirichlet series other than $\zeta(s)$.\\
In light of the above observation, we consider estimating \eqref{CAUCHYINT} in the case that $t_j := \arg(z_j)$ is small in absolute value, for each $j$. 
Indeed, 
\begin{equation*}
|z_j-\rho_j| = \rho_j|e^{it_j}-1| \leq |t_j|\rho_j,
\end{equation*}
so the argument condition above suffices in order to constrain the size of $|f_{\mbf{z}}(p)-|f_{\mbf{z}}(p)||$.
This, in turn, arises from a comparison of $F(\mbf{z};s)$ and $F(\mbf{\rho};s)$. Moreover, the sets $E_j$ may be such that the quantities $E_j(x)$ are potentially incommensurately-sized; hence, it is sensible to have separate conditions governing the behaviour of $g$ on different sets $E_j$ and $E_{j'}$. This motivates the form of Theorems \ref{HalGen} above.\\
To prove Theorem \ref{HalApp} we take $g(m) := f_{\mbf{z}}(m)$ in Theorem \ref{HalGen}. We apply \eqref{UPPER} when $|\arg(z_j)| > \eta_j$ for some $j$, and \eqref{ASYMP} when $|\arg(z_j)| \leq \eta_j$ for all $j$ in \eqref{CAUCHYINT}. The latter evaluation provides us with our main term, and the former calculation contributes to the error term. 
Given the weak nature of the constraints on our sets $E_j$, in general we cannot evaluate the sum $M_{f_{\mbf{\rho}}}(x)$ appearing in Theorem \ref{HalApp} asymptotically without further data regarding the sets (for an example of conditions in which this can be done, though, see \cite{Wir}). On the other hand, it is a particular case of a result of Selberg that when $\rho \ll 1$ as $x \ra \infty$, the sum $\sum_{n \leq x} \rho^{\omega(n)}$ can be estimated precisely (see Lemma \ref{SELBERG} below).
We may therefore invoke such estimates in the case that $\rho_j = \rho_{j'}$ for all $0 \leq j < j' \leq n$, given that $\{E_j\}_j$ partitions the primes (and thus $f_{\rho \cdot \mbf{1}}(m) = \rho^{\omega(n)}$, for any $\rho > 0$, where $\mbf{1} := (1,\ldots,1)$).\\
In order to get data about $\pi(x;\mbf{E},\mbf{k})$ at vectors some of whose components are distinct, we must extrapolate an explicit asymptotic formula for $\pi(x;\mbf{E},\mbf{k})$ from the special case where $\frac{k_j}{E_j(x)}$ is constant in $j$ to a larger collection of vectors, at least when there is some $\rho$ such that each ratio $\frac{k_j}{E_j(x)}$ is close to $\rho$. To this end, we introduce some notation. \\
Given a vector $\mbf{\rho}$ define 
\begin{align}
F(\mbf{\rho}) &:= \prod_{0 \leq j \leq n} \prod_{p\in E_j} \left(1+\frac{\rho_j}{p-1}\right)\left(1-\frac{1}{p}\right)^{\rho_j}, \label{FDef}\\
M(\mbf{\rho}) &:= \frac{1}{n+1}\sum_{0 \leq j \leq n} \rho_j. \label{MDef} 
\end{align}
Furthermore, let $b := \lim_{t \ra \infty} \left(\sum_{p \leq t} \frac{1}{p} - \log_2 t\right)$ (which exists by Mertens' theorem).
We prove the following result via an induction argument on the integers $k_j \geq \llf \rho E_j(x)\rrf+1$.
\begin{thm} \label{GOALTHM}
Assume the hypotheses of Theorem \ref{HalApp}. Furthermore, suppose there is some $\rho > 0$ such that $|\rho_j-\rho| = o\left(\min_l \log^{-1}E_l(x)	\right)$. 
Then 
\begin{equation*}
\pi(x;\mbf{E},\mbf{k}) =  \rho e^{-b(\rho-1)}\frac{F(\mbf{\rho})}{\Gamma(M(\mbf{\rho})+1)}x\prod_{0 \leq j \leq n} \frac{E_j(x)^{k_j}}{k_j!}e^{-E_j(x)}\left(1+O\left(R_j(x,\mbf{\rho},\rho)\right)\right),
\end{equation*}
where we have set
\begin{equation*}
R_j(x,\mbf{\rho},\rho) := |\rho-\rho_j|\log E_j(x) +E_j(x)^{-\frac{1}{n+1}}+ (\delta_jE_j(x))^{-\frac{1}{6}},
\end{equation*}
for each $0 \leq j \leq n$.
\end{thm}
While there is inherent ambiguity in the above lemma, as there may be more than one value of $\rho$ that satisfies the condition in the statement, it is implicit in its proof that, up to a small error, all of these values of $\rho$ give approximately the same answer. Moreover, we stress that when $\max_{0 \leq j \leq n} \rho_j$ is small (i.e., $\rho_j \ll E_j(x)^{-\e}$ for each $j$, say) then this provides an asymptotic formula that is uniform in the corresponding range of $\mbf{k}$. This is because the corresponding neighbourhoods $|\rho_j-\rho|,|\rho_j-\rho'| \ll \log^{-1}E_j(x)$ overlap for small $\rho_j$, where $\rho \neq \rho'$.\\
In the case of good partitions, Theorem \ref{GOALTHM} reproduces (in a slightly different but equivalent form) a result due to Delange \cite{Del}. \\
Applying Theorem \ref{WIRSINGEXT} ii) with $f(n) = 1$ (hence $F(s) = \zeta(s)$) and $g(n) = f_{\mbf{\rho}}(n)$ in Theorem \ref{HalApp} and $\delta, B$ close to 1 gives Corollary \ref{CORWIRSING}, generalizing Hal\'{a}sz' Poisson result to the context of a nice partition of sets $\{E_j\}_j$. The corollary is immediate upon noting that if $\rho_j := \frac{k_j}{E_j(x)}$ and using $C$ as in ii) of Theorem \ref{WIRSINGEXT},
\begin{equation*}
e^{-\sum_{0 \leq j \leq n} k_j}C = \exp\left(-\rho_jE_j(x) + \sum_{0 \leq j \leq n} \sum_{p \leq x \atop p \in E_j} \frac{\rho_j-1}{p}\right) = \exp\left(-\sum_{0 \leq j \leq n} E_j(x)\right).
\end{equation*}
\section{Auxiliary Lemmata}
The following estimate is due to Selberg.
\begin{lem}\label{SELBERG}
Let $x$ be sufficiently large, $B > 0$ and $0 < \rho \leq B$. Then 
\begin{equation} \label{SelSum}
\sum_{n \leq x} \rho^{\omega(n)} = x(\log x)^{\rho-1}F(\rho)\left(1+O_B\left(\frac{1}{\log x}\right)\right),
\end{equation}
where $F(\rho) := \Gamma(\rho)^{-1}\prod_p \left(1+\frac{\rho}{p-1}\right)\left(1-\frac{1}{p}\right)^{\rho}$. 
\end{lem}
\begin{proof}
This is proved, for instance, in Chapter II.6 of \cite{Ten2}. 
\end{proof}
Recall the definitions of $F(\mbf{\rho})$ and $M(\mbf{\rho})$ given in \eqref{FDef} and \eqref{MDef}, respectively.
\begin{lem}\label{PERTURB}
Let $\mbf{u},\mbf{v} \in (0,1/2)^{n+1}$ with $0 < \|\mbf{u}-\mbf{v}\|_1 < 1$. Then there is an absolute constant $C > 0$ such that both: i) $\left|\frac{\Gamma(M(\mbf{v})+1)}{\Gamma(M(\mbf{u})+1)} - 1\right| \leq C\|\mbf{u}-\mbf{v}\|_1$, and ii) $\left|\frac{F(\mbf{v})}{F(\mbf{u})} -1\right| \leq C\|\mbf{u}-\mbf{v}\|_1$.
\end{lem}
\begin{proof}
i) By the Weierstrass factorization theorem, $\Gamma(t) = e^{-\gamma t}t^{-1}\prod_{n \geq 1}\left(1+\frac{t}{n}\right)^{-1}e^{\frac{t}{n}}$, for $t > 0$, where $\gamma$ is the Euler-Mascheroni constant. Using the identity $\Gamma(t+1) = t\Gamma(t)$ and taking logarithms, we have
\begin{equation*}
\log \Gamma(t+1) = -\gamma t  + \sum_{m \geq 1} \left(\frac{t}{m}-\log\left(1+\frac{t}{m}\right)\right).
\end{equation*}
Set $N := \max\{\llf t \rrf,\llf r\rrf\}$. Taylor expanding the logarithm in the range $m > N$ and applying Fubini's theorem, 
\begin{align*}
\sum_{m > N} \left(\frac{t}{m}-\log\left(1+\frac{t}{m}\right)\right) = \sum_{l \geq 2} \frac{(-1)^lt^l}{l}\sum_{m > N} \frac{1}{m^l}.
\end{align*}
Assume now that $0 < r < t \leq \frac{1}{2}$ (so that $N = 0$). Subtracting $\log \Gamma(r+1)$ from $\log \Gamma(t+1)$ gives
\begin{align*}
\log\left(\frac{\Gamma(t+1)}{\Gamma(r+1)}\right) &= -\gamma(t-r) + (t-r)\sum_{l \geq 2} \frac{(-1)^l}{l}\zeta(l)\left(\sum_{0 \leq k \leq l-1} t^kr^{l-1-k}\right) \\
&= (t-r)\left(-\gamma +\sum_{l \geq 2} \frac{(-1)^l}{l}\left(\sum_{0 \leq k \leq l-1} t^kr^{l-1-k}\right) \zeta(l)\right).
\end{align*}
As $1-r > t$, 
\begin{equation*}
\sum_{0 \leq k \leq 2l-1} t^kr^{2l-1-k} - \sum_{0 \leq k \leq 2l} t^kr^{2l-k} = \sum_{0 \leq k \leq 2l-1} t^kr^{2l-1-k}(1-r) - t^{2l} > t^{2l-1}(1-r)-t^{2l} > 0.  
\end{equation*}
As $\zeta(l)$ decreases in $l$, the series above converges, and this convergence is clearly uniform in $0 < r < t < 1/2$. Let $R(r,t)$ be the corresponding value of the series. Then if $R$ is an upper bound for $R(r,t)$ that is uniform in $r$ and $t$ in this range then
\begin{equation*}
\frac{\Gamma(t+1)}{\Gamma(r+1)} = e^{(-\gamma + R)(t-r)} = 1+O\left(|R-\gamma||t-r|\right) = 1+O\left(|t-r|\right).
\end{equation*}
This implies the first claim upon taking $t := M(\mbf{v})$ and $r := M(\mbf{u})$. The case where $t < r$ is the same upon relabelling.\\
ii) For $p$ prime, let $f_p(t) := \left(1+\frac{t}{p}\right)\left(1-\frac{1}{p}\right)^t$.  Thus $F(\mbf{\rho}) := \prod_{0 \leq j \leq n} \prod_{p \in E_j} f_p(\rho_j)$. 
For $0 \leq j \leq n$ and let $p \in E_j$.  Then, as $|u_j-v_j| < 1 < p+v_j$,
\begin{align*}
f_p(u_j)f_p(v_j)^{-1} &= \left(1+\frac{u_j-v_j}{p+v_j}\right)\left(1-\frac{1}{p}\right)^{u_j-v_j} = \exp\left(\log\left(1+\frac{u_j-v_j}{p+v_j}\right)+(u_j-v_j)\log\left(1-\frac{1}{p}\right)\right) \\
&= \exp\left((u_j-v_j)\frac{u_j}{p(p+v_j)} - \sum_{l \geq 2} \frac{(u_j-v_j)^l}{l}\left(\frac{(-1)^l}{(p+v_j)^l}+\frac{1}{p^l}\right)\right) \\
&= \exp\left((u_j-v_j)\frac{u_j}{p(p+v_j)} +O\left(\frac{(u_j-v_j)^2}{p^2}\right)\right).
\end{align*}
Taking products over all primes gives
\begin{equation*}
F(\mbf{u})/F(\mbf{v}) = \exp\left(\sum_{0 \leq j \leq n} u_j(u_j-v_j)\sum_{p \in E_j} \frac{1}{p(p+v_j)} + O\left(\sum_{0 \leq j \leq n} (u_j-v_j)^2\right)\right) = 1+O\left(\sum_{0 \leq j \leq n} |u_j-v_j|\right),
\end{equation*}
which implies the second claim.
\end{proof}
We need the following lemma in order to resolve a technical difficulty involved in comparing $(\rho-1)E_j(x)$, coming from Lemma \ref{SELBERG}, with the nearest integer to $\rho E_j(x)$, which is required as input for $k_j$ in Theorem \ref{HalApp} in order to prove Theorem \ref{GOALTHM}.
\begin{lem} \label{SIMUL}
For $x \geq 2$ fixed let $0 \leq j_0 \leq n$ be such that $E_{j_0}(x) \geq E_j(x)$ for all $0 \leq j \leq n$. Then for any $\rho_0 > 0$ fixed we can choose $\rho$ such that $|\rho - \rho_0| < E_{j_0}(x)^{-1}$ and
\begin{equation*}
\max_{0 \leq j \leq n} \|\rho E_j(x)\| = \max_{0 \leq j \leq n} \min\{\{\rho E_j(x)\},1-\{\rho E_j(x)\}\} < E_{j_0}(x)^{-\frac{1}{n+1}}.
\end{equation*}
\end{lem}
\begin{proof}
Fix $\Delta := E_{j_0}(x)^{-\frac{1}{n+1}}$, define $\rho' := \rho E_{j_0}(x)$ and for $0 \leq j \leq n$ put $\alpha_j := E_j(x)/E_{j_0}(x) \in [0,1]$. Note that $\rho'$, being continuous in $\rho$, takes integer values for $x$ sufficiently large, and such integer values are bounded by $E_{j_0}(x)$. By Dirichlet's simultaneous approximation theorem (see, e.g., Theorem 200 in \cite{HaW}), we can find some $m = \rho'(m) = \rho(m)E_{j_0}(x)$ such that $\|m\alpha_j\| < \Delta$ for each $j$.  Since, by definition, $\rho(m) E_j(x) = m\alpha_j$, it follows that $\|\rho(m) E_j(x)\| < \Delta$ for each $j$. \\
Now, given an initial solution, we observe that if $\lambda := E_{j_0}(x)^{-1}$ then clearly $\|\rho(m)E_{j_0}(x)\| = \|(\rho(m)+l\lambda)E_{j_0}(x)\|$ for each $l \in \mb{Z}$, and that $\rho(m)+l\lambda= \rho(m+l)$. Since the intervals $(\rho(m) + l\lambda,\rho(m)+(l+1)\lambda]$ with $l \in \mb{Z}$ cover the interval $(0,1]$, $\rho_0$ lies within $E_{j_0}(x)^{-1}$ of some such $\rho$.
\end{proof}
In the following lemma, we apply Theorem \ref{HalGen} towards our problem.
\begin{lem}\label{APPLIC}
Suppose $g = f_{\mbf{z}}$ satisfies the conditions of Theorem \ref{HalGen}, with an implicit, nice partition $\mbf{E}$. Let $\rho_j := |z_j|$ for each $j$ and $\mbf{\rho}$ the corresponding vector of these absolute values. Then
\begin{equation*}
M_g(x)/M_{|g|}(x) \ll_{B,m} \frac{B^2}{\delta}e^{-\mc{F}(x;g)},
\end{equation*}
where, for $\beta_j := \max_{\phi \in [0,1]} |\text{arg}(z_j)-\phi|$ and $\beta := \max_{\phi \in [0,1]} \min_{0 \leq j \leq n} |\text{arg}(z_j)-\phi|$, $\gamma_{0,j} := \frac{27\pi}{1024}\beta_j^3$ and $\gamma_0 := \frac{27 \pi}{1024} \beta^3$,
\begin{equation*}
\mc{F}(x;g) := \begin{cases} \sum_{0 \leq j \leq n} \frac{\gamma_{0,j}}{2(1+\gamma_{0,j})}(z_j-\rho_j)E_j(x) & \text{ if $\mbf{E}$ is a good partition} \\ \frac{\delta}{B}\sum_{0 \leq j \leq n} \frac{\gamma_{0}}{1+\gamma_{0}}(z_j-\rho_j)E_j(x) &\text{ otherwise}. \end{cases}
\end{equation*}
\end{lem}
\begin{proof}
This amounts to a trivial verification that the sum in the exponential in \eqref{UPPER} is of the right shape in each of the cases in the statement. 
When $\mbf{E}$ is a good partition then since $\delta_j = B_j$ for each $j$, (as $g$ is constant on $E_j$) the result is immediate from Theorem \ref{HalGen} (note that this applies to any nice partition by the Pigeonhole principle, since some $E_j(x) \gg_n \log_2 x$ from the fact that $\mbf{E}$ is a partition and Mertens' theorem).
In the second case, we choose the trivial good partition, i.e., the single set consisting of all of the primes, and the resulting exponential is of the form
\begin{equation*}
\frac{\delta}{B}\frac{\gamma_{0}}{1+\gamma_{0}} \sum_{p \leq x} \frac{g(p)-|g(p)|}{p} = \frac{\delta}{B} \frac{\gamma_{0}}{1+\gamma_{0}} \sum_{1 \leq j \leq m} (z_j-\rho_j)E_j(x), 
\end{equation*}
which completes the proof.
\end{proof}
\section{Proof of Theorem \ref{HalApp} and Theorem \ref{GOALTHM}}
In this section, we prove Theorem \ref{HalApp}. Recall that we are fixing some vector $\mbf{\rho}$, the components of which we shall choose subsequently.  Given $z_j$ on the disc of radius $\rho_j$ centred at 0 in $\mb{C}$ for each $j$, we let $f_{\mbf{z}}(m) := \prod_{0 \leq j \leq n} z_j^{\omega_{E_j}(m)}$.  
\begin{proof}[Proof of Theorem \ref{HalApp}]
Rewriting \eqref{CAUCHYINT} by making the change of variable $z_j = \rho_je^{it_j}$, we seek to estimate the integral
\begin{equation*}
\pi(x;\mbf{E},\mbf{k}) = (2\pi)^{-(n+1)}\left(\prod_{0 \leq j \leq n} \rho_j^{-k_j}\right) \int_{[-\pi,\pi]^{n+1}} d\mbf{t}e^{-i\mbf{k}\cdot \mbf{t}} \left(\sum_{m \leq x} f_{\mbf{z}}(m)\right).
\end{equation*}
Fix $\theta_0,\ldots,\theta_n \in (0,1)$ and let $\mc{B} := \prod_{0 \leq j \leq n} [-\theta_j,\theta_j]$ and $\mc{E} := [-\pi,\pi]^{n+1}\bk \mc{B}$. Let $I_{\mc{B}}$ denote the multivariable integral defined over $\mc{B}$ and $I_{\mc{E}}$ the corresponding integral over $\mc{E}$.  Let us estimate $I_{\mc{E}}$ first. For convenience, put also $J_{\mc{B}} := (2\pi)^{-(n+1)} \left(\prod_{0 \leq j \leq n} \rho_j^{-k_j}\right)I_{\mc{B}}$, which we will use later.  Set now $U^{0}_j := [-\theta_j,\theta_j]$ and $U^1_j := [-\pi,\pi]\bk [-\theta_j,\theta_j]$.  For a vector $\mbf{s} \in \{0,1\}^{n+1}$, let $U_{\mbf{s}} := \prod_{0 \leq j \leq n} U_j^{s_j}$.  Then clearly, $\mc{E} = \bigcup_{\mbf{s} \in \{0,1\}^{n+1} \atop \mbf{s} \neq \mbf{0}} U_{\mbf{s}}$, and since these sets are disjoint, we have moreover $I_{\mc{E}} = \sum_{\mbf{s} \in \{0,1\}^{n+1} \bk{\mbf{0}}} I_{\mbf{s}}$, where
\begin{equation*}
I_{\mbf{s}} := \int_{U_{\mbf{s}}} d\mbf{t}e^{-i\mbf{k}\cdot \mbf{t}} \left(\sum_{m \leq x} f_{\mbf{z}}(m)\right).
\end{equation*}
Given $\mbf{s} \neq \mbf{0}$ and $\mbf{t} \in U_{\mbf{s}}$, we can choose $\phi_j$ such that the map $\gamma \mapsto |\gamma - t_j|$ is minimized for each $j$ in the case that $\mbf{E}$ is good, and $\phi$ to be the minimizing choice for the map $\gamma \mapsto \min_{0 \leq j \leq n}|\gamma - t_j|$ otherwise. Let $\beta_j$, respectively $\beta$, be the corresponding minimal values.  
Set $\delta := \min_{0 \leq j \leq n} \rho_j$ and $B := \max_{0 \leq j \leq n} \rho_j$. Applying \eqref{UPPER}, we get
\begin{equation*}
I_{\mbf{s}} \ll_{B,n} \frac{B^2}{\delta}\left(\sum_{m \leq x} f_{\mbf{\rho}}(m) \right)\int_{U_{\mbf{s}}} d\mbf{t} e^{-\sum_{0 \leq j \leq n} D_j\rho_j(1-\cos t_j)E_j(x)},
\end{equation*}
where $f_{\mbf{\rho}}(m) = \prod_{0 \leq j \leq n} \rho_j^{\omega_{E_j}(m)}$ and for each $j$, 
\begin{equation*}
D_j := \begin{cases} \frac{\gamma_{0,j}}{2(1+\gamma_{0,j})} &\text{ if $\mbf{E}$ is good} \\ \frac{\delta}{B}\frac{\gamma_{0}}{2(1+\gamma_{0})} &\text{ otherwise.}\end{cases}
\end{equation*}
Set $\Theta_j := \left(\frac{1}{3}D_j\rho_jE_j(x)\right)^{\frac{1}{2}}\theta_j$. Then, as the multivariate integral above splits into $n+1$ integrals in its respective variables, we have
\begin{equation*}
\int_{t_j \in U_j^{s_j}} dt_j e^{-\sum_j D_j\rho_j(1-\cos t_j)E_j(x)} \leq \int_{|t_j| > \theta_j} dt_j e^{-\frac{1}{3}D_jt_j^2 \rho_jE_j(x)} = \left(\frac{1}{3}D_j\rho_jE_j(x)\right)^{-\frac{1}{2}} \int_{|t_j| > \Theta_j} dt_j e^{-t_j^2} \ll \theta_je^{-\Theta_j^2},
\end{equation*}
this last estimate following by choosing an optimal $\alpha > 0$, in this case $\alpha = \Theta_j^{-2}$, such that 
\begin{equation*}
\int_{|t_j| > \Theta_j} e^{-t_j^2} dt_j \leq e^{-(1-\alpha)\Theta_j^2} \int_{-\infty}^{\infty} e^{-\alpha t_j^2} dt_j \asymp \alpha^{-\frac{1}{2}}e^{-(1-\alpha)\Theta_j^2}.
\end{equation*}
It follows that
\begin{align*}
I_{\mc{E}} &= \sum_{\mbf{s} \in \{0,1\}^{n+1} \bk \{\mbf{0}\}} I_{\mbf{s}} \ll_{B,n} \frac{B^2}{\delta}\left(\sum_{m \leq x} f_{\mbf{\rho}}(m) \right)\left(\prod_{0 \leq j \leq n} \frac{1}{3}D_j\rho_jE_j(x)\right)^{-\frac{1}{2}}\prod_{0 \leq j \leq n} \Theta_je^{-\frac{1}{3} D_j\rho_jE_j(x)\theta_j^2}.
\end{align*}
We now apply \eqref{ASYMP} to $I_{\mc{B}} = I_{\mbf{0}}$ with $\eta_j = \theta_j$ for each $j$. Write $R_{\mbf{z}} := \left(\sum_{m \leq x} (f_{\mbf{z}}(m) - A_{\mbf{z}}f_{\mbf{\rho}}(m))\right)$. Note that $A_{\mbf{z}} = \exp\left(\sum_{0 \leq j \leq n} (z_j-\rho_j) \sum_{p \leq x \atop p \in E_j} \frac{1}{p}\right)$. Further, let $\mc{S} := \{\mbf{z} \in \mb{C}^{n+1} : |z_j| = \rho_j, \arg(\mbf{z}) \in \mc{B}\}$, where $\arg(\mbf{z}) := (\arg(z_0),\ldots,\arg(z_n))$. Then, recalling that $\mc{D} := \prod_{0 \leq j \leq n} \partial D\left(0,\rho_j\right)$, 
\begin{align*}
&J_{\mc{B}} = \left(\sum_{m \leq x} f_{\mbf{\rho}}(m)\right) \frac{1}{(2\pi i)^{n+1}}\int_{\mbf{z} \in \mc{S}} \frac{d\mbf{z}}{z_0^{k_0+1}\cdots z_n^{k_n+1}}\exp\left(\sum_{0 \leq j \leq n} (z_j-\rho_j)E_j(x)\right) + O\left(\int_{\mbf{t} \in \mc{B}} d\mbf{t} |R_{\mbf{z}}|\right) \\
&= \left(\sum_{m \leq x} f_{\mbf{\rho}}(m)\right) \left((2\pi i)^{-(n+1)} \int_{\mc{D}} \frac{d\mbf{z}}{z_0^{k_0+1}\cdots z_n^{k_n+1}}\exp\left(\sum_{0 \leq j \leq n} (z_j-\rho_j)E_j(x)\right) + O\left(\mc{R}_1\right)\right) + O\left(\mc{R}_2\right),
\end{align*}
where we have defined $\mc{R}_2 := \int_{\mbf{t} \in \mc{B}} d\mbf{t} |R_{\mbf{z}}|$ and
\begin{align*}
\mc{R}_1 &:= (2\pi)^{-(n+1)} \left(\prod_{0 \leq j \leq n} \rho_j^{-k_j}\right)\left|\int_{\mbf{t} \notin \mc{E}} d\mbf{t}\exp\left(-\sum_{0 \leq j \leq n} \rho_jE_j(x)(1-e^{it_j})\right)\right| \\
&\leq (2\pi)^{-(n+1)} \left(\prod_{0 \leq j \leq n} \rho_j^{-k_j}\right)\int_{\mbf{t} \notin \mc{E}} d\mbf{t}\exp\left(-\frac{1}{3}\sum_{0 \leq j \leq n} \rho_jE_j(x)t_j^2\right).
\end{align*}
It is an obvious consequence of the Taylor expansion of $e^z$ that
\begin{align*}
\frac{1}{(2\pi i)^{n+1}}\int_{\mc{D}}  d\mbf{z}\frac{e^{\sum_{0 \leq j \leq n} (z_j-\rho_j)E_j(x)}}{z_0^{k_0+1}\cdots z_n^{k_n+1}} &= \prod_{0 \leq j \leq n}\left(\frac{1}{2\pi i} \int_{\partial D(0,\rho_j)} \frac{dz_j}{z_j^{k_j+1}}e^{(z_j-\rho_j)E_j(x)}\right) = \prod_{0 \leq j \leq n} \frac{E_j(x)^{k_j}}{k_j!}e^{-\rho_jE_j(x)},
\end{align*}
which constitutes the main term of $\pi(x;\mbf{E},\mbf{k})$. \\
In $\mc{R}_1$, each of the sets $U_{\mbf{s}}$ has $s_j = 1$ for some index $j$. In this case, 
\begin{equation*}
\int_{|t_j| > \theta_j} e^{-\frac{1}{3}\rho_jE_j(x)t_j^2} dt_j \ll (\rho_jE_j(x))^{-\frac{1}{2}}\theta_je^{-\frac{1}{2}\rho_jE_j(x)\theta_j^2}. 
\end{equation*}
For the remaining indices $j' \neq j$ we bound trivially by the full Gaussian integral, which has value $\ll (\rho_{j'}E_{j'}(x))^{-\frac{1}{2}}$ (as calculated below). Summing over the $\ll_n 1$ sets in $\mc{E}$, we get
\begin{align*}
|\mc{R}_1| \ll_{B,n} (2\pi)^{-(n+1)} \left(\prod_{0 \leq j \leq n} \rho_j^{-k_j}\right) \left(\prod_{0 \leq j \leq n} \rho_jE_j(x)\right)^{-\frac{1}{2}} \sum_{0 \leq l \leq n} \theta_le^{-\frac{1}{2}\rho_lE_l(x)\theta_l^2} =: (2\pi)^{-(n+1)}\left(\prod_{0 \leq j \leq n} \rho_j^{-k_j}\right)\mc{R}'_1.
\end{align*}
Write $\mc{R}_2 := (2\pi)^{-(n+1)}\prod_{0 \leq j \leq n} \rho_j^{-k_j} \mc{R}'_2$. Now, from Theorem \ref{HalGen} ii) (upon taking $d_1 = 1$)
\begin{equation*}
|R_{\mbf{z}}| \ll_{B,n} \frac{B^2}{\delta}\left(\sum_{m \leq x} f_{\mbf{\rho}}(m)\right) \left(\eta^{\frac{1}{2}} \mc{I}_1 + \left(\log^{-\pi^3 \delta/2} x + \frac{1}{\delta}e^{-\frac{1}{\sqrt{\eta}}}\right)\mc{I}_2\right),
\end{equation*}
where we have defined $\mc{I}_1 := \int_{\mc{B}} |A_{\mbf{z}}| d\mbf{t}$ and $\mc{I}_2 := \int_{\mc{B}} |A_{\mbf{z}}|^{\frac{\gamma_0}{2(1+\gamma_0)}} d\mbf{t}$. These integrals are evaluated in essentially the same way, so it suffices to consider $\mc{I}_1$ and indicate how the result changes with $\mc{I}_2$.  As before,
\begin{align*}
\mc{I}_1 &= \prod_{0 \leq j \leq n} \int_{-\theta_j}^{\theta_j} e^{-\rho_j(1-\cos t_j)E_j(x)} dt_j \leq \prod_{0 \leq j \leq n} \int_{-\theta_j}^{\theta_j} e^{-\frac{1}{3}\rho_jE_j(x)t_j^2} dt_j \\
&\ll \prod_{0 \leq j \leq n} \frac{1}{\sqrt{\rho_jE_j(x)}}\int_{-\infty}^{\infty} e^{-t_j^2} dt_j = \sqrt{\frac{\pi}{\rho_jE_j(x)}}.
\end{align*}
In $\mc{I}_2$, the expression in the exponential is simply $\frac{\gamma_0}{2(1+\gamma_0)}\rho_j(1-\cos t_j)$, and since $\gamma_0 \gg 1$, the same estimate as for $\mc{I}_1$ occurs for $\mc{I}_2$.  It thus follows that
\begin{equation*}
|\mc{R}'_2| \ll_{B,n} \frac{B^2}{\delta}\left(\sum_{m \leq x} f_{\mbf{\rho}}(m)\right) \left(\eta^{\frac{1}{2}} + \left(\log^{-\pi^3 \delta/2} x + \frac{1}{\delta}e^{-\frac{1}{\sqrt{\eta}}}\right)\right)\prod_{0 \leq j \leq n} \sqrt{\frac{\pi}{\rho_jE_j(x)}}.
\end{equation*}
The estimates for $|\mc{R}'_1|$ exceeding those of $|I_{\mc{E}}|$, we find that
\begin{align*}
&|I_{\mc{E}}| + |\mc{R}'_1| + |\mc{R}'_2| \ll_{B,n} \frac{B^2}{\delta}\prod_{0 \leq j \leq n} \left(\frac{\pi}{\rho_jE_j(x)}\right)^{\frac{1}{2}} M_{f_{\mbf{\rho}}}(x)\left(\eta^{\frac{1}{2}} + \log^{-\pi^3\delta/2} x + \frac{e^{-\frac{1}{\sqrt{\eta}}}}{\delta} + \sum_{0 \leq l \leq n} \theta_le^{-\frac{1}{3}D_j\rho_lE_l(x)\theta_l^2}\right).
\end{align*}
Collecting these estimates, we have
\begin{align*}
&\pi(x;\mbf{E},\mbf{k}) = J_{\mc{B}} + O\left((2\pi)^{-(n+1)}\left(\prod_{0 \leq j \leq n} \rho_j^{-k_j}\right) \left(|I_{\mc{E}}| + |\mc{R}'_1| + |\mc{R}'_2|\right)\right) \\
&= M_{f_{\mbf{\rho}}}(x)\left(\prod_{0 \leq j \leq n} \frac{E_j(x)^{k_j}}{k_j!}e^{-\rho_jE_j(x)} + O_{B,n}\left(\mc{R}\right)\right),
\end{align*}
where
\begin{equation*}
\mc{R} := \frac{B^2}{\delta} \left(\prod_{0 \leq j \leq n} \rho_j^{k_j}(2\pi \rho_jE_j(x))^{\frac{1}{2}}\right)^{-1}\left(\eta^{\frac{1}{2}} + \log^{-\pi^3\delta/2} x + \frac{1}{\pi^3\delta}e^{-\frac{1}{\sqrt{\eta}}} + \sum_{0 \leq l \leq n} \theta_le^{-\frac{1}{3}D_j\rho_lE_l(x)\theta_l^2}\right).
\end{equation*}
Now set $\rho_j := \frac{k_j}{E_j(x)}$. By Stirling's formula, 
\begin{equation*}
k_j! = (2\pi \rho_jE_j(x))^{\frac{1}{2}}\rho_j^{k_j}E_j(x)^{k_j}e^{-\rho_jE_j(x)}\left(1+O\left(\frac{1}{\rho_jE_j(x)}\right)\right).
\end{equation*}
Hence, $\frac{E_j(x)^{k_j}}{k_j!}e^{-k_j} = (2\pi \rho_jE_j(x))^{-\frac{1}{2}}\rho_j^{-k_j}\left(1+O\left(\frac{1}{\rho_jE_j(x)}\right)\right)$. It then follows that since $\rho_jE_j(x) \gg 1$,
\begin{equation*}
\mc{R} \ll \prod_{0 \leq j \leq n}\frac{E_j(x)^{k_j}}{k_j!}e^{-k_j}\left(\eta^{\frac{1}{2}} + \log^{-\pi^3\delta/2} x + \frac{1}{\pi^3\delta}e^{-\frac{1}{\sqrt{\eta}}} + \theta_le^{-\frac{1}{3}D_l\rho_lE_l(x)\theta_l^2}\right),
\end{equation*}
whence follows the expression
\begin{equation}
\pi(x;\mbf{E},\mbf{k}) = M_{f_{\mbf{\rho}}}(x)\prod_{0 \leq j \leq n} \frac{E_j(x)^{k_j}}{k_j!}e^{-k_j}\left(1+O_{B,n}\left(\frac{B^2}{\delta}\left(\eta^{\frac{1}{2}} + \log^{-\pi^3\delta/2} x + \frac{1}{\delta}e^{-\frac{1}{\sqrt{\eta}}} + \mc{L}_j\right)\right)\right), \label{ALMOSTDONE}
\end{equation}
where $\mc{L}_j := \theta_je^{-\frac{1}{3}D_j\rho_jE_j(x)\theta_j^2}$.
It remains to select $\theta_j$ optimally for each $j$. Recall that, in this context, $\eta = \max_{0 \leq j \leq n} \theta_j$.  We consider two cases, according to which of the first and third expressions in the error term in \eqref{ALMOSTDONE} (which are the ones depending on $\theta_j$) is largest. \\
Suppose first that $\eta^{\frac{1}{2}}$ is largest. It suffices to choose $\theta_j$ such that $\theta_je^{-\frac{1}{6}D_j\rho_jE_j(x)\theta_j^2} \ll 1$ for each $j$. For sufficiently large $x$, we may then take $\theta_j =(D_j	\rho_jE_j(x))^{-\frac{1}{3}}$, whenever $\rho_j \gg E_j(x)^{-1+\e}$, as we are assuming. \\
Next, suppose that $\frac{1}{\delta}e^{-\frac{1}{\sqrt{\eta}}}$ is largest, where we recall that $\delta := \min_{0 \leq j \leq n} \rho_j$.  Choose $\theta_j$ such that $\frac{1}{3}D_j\rho_jE_j(x)\theta_j^{\frac{5}{2}} = 2$, i.e., take $\theta_j := \left(\frac{D_j}{6}\rho_jE_j(x)\right)^{-\frac{2}{5}}$.  It then follows that $e^{-\frac{1}{3}D_j\rho_jE_j(x)\theta_j^2} \ll e^{-\frac{2}{\sqrt{\eta}}}$. 
Now, in each of these cases we have, uniformly over all $\rho_j$ in our range,
\begin{align*}
\eta^{\frac{1}{2}} &\leq \sum_{0 \leq j \leq n} \theta_j^{\frac{1}{2}} \ll \sum_{0 \leq j \leq n} \left((D_j\rho_jE_j(x))^{-\frac{1}{6}} + (\rho_jE_j(x))^{-\frac{1}{5}}\right)  \ll \sum_{0 \leq j \leq n} \left(D_j\delta_jE_j(x)\right)^{-\frac{1}{6}} \\
\frac{1}{ \delta} e^{-\frac{1}{\sqrt{\eta}}} &\ll \frac{1}{\delta}\sum_{0 \leq j \leq n}\left(e^{-(D_j\rho_jE_j(x))^{-\frac{1}{6}}} + e^{-(D_j\rho_jE_j(x))^{-\frac{1}{5}}}\right) \ll \frac{1}{\delta}\sum_{0 \leq j \leq n}\left(D_j\delta_jE_j(x)\right)^{-\frac{1}{6}},
\end{align*}
the last estimate occurring trivially for sufficiently large $x$.  Thus, it follows that
\begin{equation*}
\pi(x;\mbf{E},\mbf{k}) = M_{f_{\mbf{\rho}}}(x)\prod_{0 \leq j \leq n} \frac{E_j(x)^{k_j}}{k_j!}e^{-k_j}\left(1+O_{B,n}\left(\left(\frac{B}{\delta}\right)^2\left((D_j\delta_jE_j(x))^{-\frac{1}{6}} + \log^{-\pi^3\delta/2} x\right)\right)\right), 
\end{equation*}
which completes the proof upon using the definition of $D_j$.
\end{proof}
By Lemma \ref{SELBERG}, we know that when $\rho_j = \rho$ for all $0 \leq j \leq n$ (in which case $B = \delta = \rho$), we have
\begin{align*}
M_{f_{\mbf{\rho}}}(x) &= \sum_{m \leq x} \rho^{\omega(m)} = \left(1+O_{\rho}\left(\frac{1}{\log x}\right)\right)\frac{F(\rho)}{\Gamma(\rho)}x\log^{\rho-1} x \\
&= \left(1+O_{\rho}\left(\frac{1}{\log x}\right)\right)\frac{F(\rho)}{\Gamma(\rho)}x\exp\left((\rho-1)\sum_{0 \leq j \leq n} E_j(x) -(\rho-1)b\right),
\end{align*}
where $b := \lim_{t \ra \infty} \left(\sum_{p \leq t} \frac{1}{p}-\log_2 t\right)$.  Now, suppose that $\rho_0$ is such that $|\rho_j-\rho_0| = o\left(\min_l \log^{-1}E_l(x)\right)$, as in the hypotheses of Theorem \ref{GOALTHM}. By Lemma \ref{SIMUL} we may select $\rho$ such that $\max_{0 \leq j \leq n} \|\rho E_j(x)\| \ll E_{j_0}(x)^{-1/(n+1)}$, with $|\rho-\rho_0| < E_{j_0}(x)^{-1}$, where $j_0$ is the index for which $E_{j_0}(x)$ is maximal. For this choice, put $\mbf{k} = \mbf{k}(\rho)$, where each $k_j(\rho)$ is the nearest integer to $\llf\rho E_j(x)\rrf$. Then
\begin{align*}
\pi(x;\mbf{E},\mbf{k}) &= xe^{-b(\rho-1)}\frac{F(\rho)}{\Gamma(\rho)}\prod_{0 \leq j \leq n} \frac{E_j(x)^{k_j(\rho)}}{k_j(\rho)!}e^{(\rho_j-1)E_j(x)-k_j(\rho)}\left(1+O_{B,n}\left((\delta_jE_j(x))^{-\frac{1}{6}} + \log^{-\min\{\rho,\pi^3\delta/2\}} x\right)\right) \\
&= xe^{-b(\rho-1)}\frac{F(\rho)}{\Gamma(\rho)}\prod_{0 \leq j \leq n} \frac{E_j(x)^{k_j(\rho)}}{k_j(\rho)!}e^{-E_j(x)+(\rho E_j(x)-k_j(\rho)) }\left(1+O_{B,n}\left((\delta_jE_j(x))^{-\frac{1}{6}} + \log^{-\pi^3\delta/2} x\right)\right) \\
&= xe^{-b(\rho-1)}\frac{F(\rho)}{\Gamma(\rho)}\prod_{0 \leq j \leq n} \frac{E_j(x)^{k_j(\rho)}}{k_j(\rho)!}e^{-E_j(x)}\left(1+O_{B,n}\left((\delta_jE_j(x))^{-\frac{1}{6}} + \log^{-\pi^3\delta/2} x + E_j(x)^{-\frac{1}{n+1}}\right)\right)
\end{align*}
(note that $\frac{B}{\delta} = 1$ in this case). Since $|\rho-\rho_0| < E_{j_0}(x)^{-1}$, it suffices, in order to prove Theorem \ref{GOALTHM}, to consider $|k_j/E_j(x)-\rho| \ll \log^{-1}E_j(x)$. \\
To deduce Theorem \ref{GOALTHM} from Theorem \ref{HalApp} we first prove the following induction lemma. 
\begin{lem} \label{INDUCT}
Let $\rho > \max_j \delta_j$. For $k_j \in \mb{N}$ and $\rho_j := k_j/E_j(x)$, suppose $\left|\rho_j-\rho\right| = o\left(1/\log E_j(x)\right)$ for each $j$. Then
\begin{equation*}
\pi(x;\mbf{E},\mbf{k}) = e^{-b(\rho-1)}\frac{F(\rho)}{\Gamma(\rho)}x\prod_{0 \leq j \leq n} \frac{E_j(x)^{k_j}}{k_j!}e^{-E_j(x)} \left(1+O\left(\frac{\log E_j(x)}{E_j(x)}|k_j-k_j(\rho)| + (\delta_jE_j(x))^{-\frac{1}{6}}+E_j(x)^{-\frac{1}{n+1}}\right)\right).
\end{equation*}
\end{lem}
\begin{proof}
Consider first the case that $\rho_j \geq \rho$ for each $j$. We shall prove this by induction on $k_j$ in each $j$, under the constraint $|k_j-k_j(\rho)| \ll E_j(x)/\log E_j(x)$. The base case, with $k_j(\rho)$ is verified immediately as a consequence of Theorem \ref{HalApp}. \\
Assume that the result holds for $k_j > k_j(\rho)$. Call the error term from the inductive hypothesis $r_{j,k_j}(x)$. Let $y := e^{\frac{\log x}{E_{j}(x)}}$.  Then, as we can select any of $k_j+1$ prime power factors from $E_j$ from the factorizations of integers counted by $\pi(x;\mbf{E},\mbf{k})$,
\begin{align}
&e^{b(\rho-1)}\Gamma(\rho)F(\rho)^{-1}\pi\left(x;\mbf{E},\mbf{k}+\mbf{e}_j\right) = \frac{e^{b(\rho-1)}\Gamma(\rho)F(\rho)^{-1}}{k_j+1}\left(\sum_{p^{\nu} \leq x^{\frac{1}{2}} \atop p \in E_j} \pi\left(x/p^{\nu};\mbf{E},\mbf{k}\right) +O\left(x^{\frac{1}{2}}\right)\right)\nonumber\\
&= \frac{x}{k_j+1}\left(\sum_{p^{\nu} \leq x^{\frac{1}{2}} \atop p \in E_j} \frac{1}{p^{\nu}}\prod_{0 \leq l \leq n} \frac{E_j\left(x/p^{\nu}\right)^{k_j}}{k_j!}e^{-E_j(x/p^{\nu})}\left(1+O\left(r_{j,k_j}\left(x^{\frac{1}{2}}\right)\right)\right) +O\left(x^{-\frac{1}{2}}\Gamma(\rho)F(\rho)^{-1}\right)\right) \nonumber\\
&= \frac{x}{k_j+1}(\prod_{0 \leq l \leq n}\left(1+O\left(r_{j,k_j}\left(x^{\frac{1}{2}}\right)\right)\right) \frac{E_l(x)^{k_l}}{k_l!}e^{-E_l(x)} \nonumber\\
&\cdot \left(\sum_{p^{\nu} \leq y \atop p \in E_j} + \sum_{y < p^{\nu} \leq x^{\frac{1}{2}}}\right)\prod_{0 \leq j \leq n} \left(1-\frac{E_l(x)-E_l(x/p^{\nu})}{E_l(x)}\right)^{k_l} \frac{\left(1+O\left(\frac{1}{\log x} \right)\right)}{\left(1-\frac{\log p^{\nu}}{\log x}\right)} +O\left(x^{-\frac{1}{2}}\Gamma(\rho)F(\rho)^{-1}\right))\nonumber\\
&= \frac{x}{k_j+1}\left(\prod_{0 \leq l \leq n}\left(1+O\left(r_{j,k_j}\left(x^{\frac{1}{2}}\right)+\log^{-1} x\right)\right) \frac{E_l(x)^{k_l}}{k_l!}e^{-E_l(x)} \left(\Sigma_1 + \Sigma_2\right) + O\left(x^{-\frac{1}{2}}\Gamma(\rho)F(\rho)^{-1}\right)\right)\label{DECOMP},
\end{align}
the second last estimate following because, in replacing $e^{-E_j(x/p^{\nu})}$ by $e^{-E_j(x)}$, we produce the term
\begin{equation*}
\exp\left(\sum_{0 \leq l \leq n} \left(E_l(x)-E_l(x/p^{\nu})\right)\right) = \frac{\log x}{\log x-\log p^{\nu}}+ O\left(\frac{1}{\log(x/p^{\nu})}\right),
\end{equation*}
with $p^{\nu} \leq x^{\frac{1}{2}}$. Here, we have set 
\begin{align*}
\Sigma_1 &:= \sum_{p^{\nu} \leq y \atop p \in E_j} \prod_{0 \leq l \leq n} \left(1-\frac{E_l(x)-E_l(x/p^{\nu})}{E_l(x)}\right)^{k_l}\left(1-\frac{\log p^{\nu}}{\log x}\right)^{-1}, \\
\Sigma_2 &:= \sum_{y < p^{\nu} \leq \sqrt{x} } \prod_{0 \leq l \leq n} \left(1-\frac{E_l(x)-E_l(x/p^{\nu})}{E_l(x)}\right)^{k_l}\left(1-\frac{\log p^{\nu}}{\log x}\right)^{-1}.
\end{align*}
Consider $\Sigma_1$. As $x/y = x^{1-\alpha}$ with $\alpha = \frac{1}{E_{j}(x)}$, we get 
\begin{align*}
\left(1-\frac{E_j(x)-E_j(x/p^{\nu})}{E_j(x)}\right)^{k_j} &\leq k_j\left(\frac{E_j(x)-E_j(x/y)}{E_j(x)}\right) \leq k_jE_j(x)^{-1}\sum_{x/y < p \leq x} \frac{1}{p} \\
&\leq k_j \frac{\left|\log\left(1-\frac{1}{E_{j}(x)}\right)\right|}{E_j(x)} \leq \frac{k_j}{E_{j}(x)^2}.
\end{align*}
Thus, as $\left(1-\frac{\log p^{\nu}}{\log x}\right)^{-1} = \left(1+O\left(\frac{\log y}{\log x}\right)\right) = 1+O\left(\frac{1}{E_{j}(x)}\right)$, which majorizes the remaining error terms, 
\begin{equation*}
\Sigma_1 = \left(1+O\left(\frac{1}{E_{j}(x)}\right)\right)\sum_{p^{\nu} \leq y \atop p \in E_j} \frac{1}{p^{\nu}} = \left(1+O\left(\frac{1}{E_{j}(x)}\right)\right)E_j(y) = \left(1+O\left(\frac{1}{E_{j}(x)} + \frac{\log E_{j}(x)}{E_j(x)}\right)\right)E_j(x),
\end{equation*}
this last estimate coming from $E_j(x) = E_j(y) + O\left(\log E_{j}(x)\right)$, and $\sum_{p \leq y \atop \nu \geq 2} \frac{1}{p^{\nu}} = O(1)$. \\
Next, for $\Sigma_2$, we have $\prod_{0 \leq j \leq n} \left(\frac{E_j(x/p)}{E_j(x)}\right)^{k_j} \leq 1$ and $1-\frac{\log p}{\log x} \geq \frac{1}{2}$, so that 
\begin{equation*}
\Sigma_2 \leq 2\sum_{y < p^{\nu} \leq x^{\frac{1}{2}} \atop p^{\nu} \in E_j} \frac{1}{p^{\nu}} \ll \log E_{j}(x),
\end{equation*}
as before.  All told, this implies that
\begin{equation*}
\Sigma_1+\Sigma_2 = \left(1+O\left(\frac{\log E_{j}(x)}{E_j(x)}\right)\right)E_j(x) + O\left(\log E_{j}(x)\right) = \left(1+O\left(\frac{\log E_j(x)}{E_j(x)}\right)\right)E_j(x).
\end{equation*}
Inserting this into \eqref{DECOMP}, we see that $\pi(x;\mbf{E},\mbf{k}+\mbf{e}_j) = \frac{E_j(x)}{k_j+1}\left(1+O\left(r_{j,k_j}(x) + \log^{-1} x + \frac{\log E_j(x)}{E_j(x)}\right)\right)\pi(x;\mbf{E},\mbf{k})$, so that $r_{j,k_j+1}(x) = \frac{\log E_j(x)}{E_j(x)}+r_{j,k_j}(x)$ is admissible. Iterating $k_j-k_j(\rho)$ for each $j$ gives the result when $\rho_j \geq \rho$ for each $j$.
In the case that $\rho_j < \rho$ for some $j$ then we may replace $\rho$ with any $E_j(x)^{-1+\e} \ll \rho' < \rho_j$. Then $|\rho'-\rho| = o(\min_l \log^{-1} E_l(x))$, so that the same inductive argument as above holds from $\rho'$ to $\rho_j$ for each $j$.\\
\end{proof}
%
\begin{proof}[Proof of Theorem \ref{GOALTHM}]
We apply the estimates $F(\mbf{\rho}) = F(\rho)\left(1+O\left(\sum_{0 \leq j \leq n} |\rho_j-\rho|\right)\right)$ and 
\begin{equation*}
\Gamma(M(\mbf{\rho})+1) = \left(1+O\left(\sum_{0 \leq j \leq n} |\rho_j-\rho|\right)\right)\Gamma(\rho+1) = \rho\Gamma(\rho)\left(1+O\left(\sum_{0 \leq j \leq n} |\rho_j-\rho|\right)\right)
\end{equation*}
from Lemma \ref{PERTURB} in Lemma \ref{INDUCT}, with $\rho_j= \frac{k_j}{E_j(x)}$ as before. Noting that $|\rho-\rho_j| = \frac{1}{E_j(x)}|k_j-k_j(\rho)|$, it follows that
\begin{equation*}
\pi(x;\mbf{E},\mbf{k}) = \rho \frac{e^{-b(\rho-1)}F(\mbf{\rho})}{\Gamma(M(\mbf{\rho})+1)}x\prod_{0 \leq j \leq n} \frac{E_j(x)^{k_j}}{k_j!}e^{-E_j(x)} \left(1+O\left(\frac{\log E_j(x)}{E_j(x)}|k_j-k_j(\rho)| + (\delta_jE_j(x))^{-\frac{1}{6}}+E_j(x)^{-\frac{1}{n+1}}\right)\right).
\end{equation*}
This completes the proof.
\end{proof}
\section*{Acknowledgements}
The author would like to thank his Ph.D supervisor Dr. J. Friedlander for his ample patience and encouragement during the period in which this paper was written.  

\bibliographystyle{plain}
\bibliography{bibHalApp}
\end{document}